\documentclass{amsart}
\usepackage{amsmath,amsthm,amsfonts,amssymb}
\usepackage{mathtools}
\usepackage{graphicx} 
\usepackage{comment}
\usepackage{enumitem}
\usepackage[hidelinks]{hyperref}
\usepackage{xcolor}

\title[Quasisymmetric mappings on variants of fractal percolation]{Quasisymmetric mappings on two variants of fractal percolation}
\author{Roope Anttila}
\address[Roope Anttila]
        {University of St Andrews \\ 
         Mathematical Institute\\ 
         St Andrews, KY16 9SS\\ 
         Scotland}
\email{ra216@st-andrews.ac.uk}

\author{Sylvester Eriksson-Bique}
\address[Sylvester Eriksson-Bique]{Department of Mathematics and Statistics\\  P.O. Box 35, FI-40014, University of Jyväskylä,
 Finland }
\email{sylvester.d.eriksson-bique@jyu.fi}

\author{Aleksi Pyörälä}
\address[Aleksi Pyörälä]{Department of Mathematics and Statistics\\  P.O. Box 35, FI-40014, University of Jyväskylä,
 Finland }
\email{aleksi.pyorala@gmail.com}

\thanks{RA was financially supported the Magnus Ehrnrooth foundation and EPSRC, grant no. EP/Z533440/1. SEB was supported by the Research Council of Finland via the project \emph{GeoQuantAM: Geometric and Quantitative Analysis on Metric spaces}, grant no. 354241. AP was supported by the Research Council of Finland via grants 354241 and 355453.}

\newcommand{\PP}{\mathbb{P}}

\newcommand{\E}{\mathcal{E}}

\newcommand{\F}{\mathcal{F}}
\newcommand{\cond}{\ \big|\ }
\newcommand{\N}{\mathbb{N}} 
\newcommand{\R}{\mathbb{R}}
\newcommand{\mtt}[1]{\mathtt{#1}}
\newtheorem{theorem}{Theorem}[section]
\newtheorem{proposition}[theorem]{Proposition}
\newtheorem{corollary}[theorem]{Corollary}

\newtheorem{lemma}[theorem]{Lemma}
\newtheorem{question}{Question}

\DeclareMathOperator{\dist}{dist}
\DeclareMathOperator{\dima}{dim_A}
\DeclareMathOperator{\dimh}{dim_H}
\DeclareMathOperator{\Cdimh}{\mathcal{C}dim_H}

\begin{document}
\begin{abstract}
    We study quasisymmetric maps on two variants of the classical fractal percolation model: the fat and dense fractal percolations. We show that, almost surely conditioned on non-extinction, the Hausdorff dimension of the fat fractal percolation cannot be lowered with a quasisymmetry and the Hausdorff dimension of the dense fractal percolation cannot be lowered with a power quasisymmetry.
\end{abstract}
\maketitle

\section{Introduction}
For a given homeomorphism $\eta\colon [0,\infty)\to[0,\infty)$, a function $f\colon X\to Y$ between metric spaces $(X,d)$ and $(Y,\rho)$ is called an \emph{$\eta$-quasisymmetry} if
\begin{equation*}
    \frac{\rho(f(x),f(y))}{\rho(f(x),f(z))}\leq \eta\left(\frac{d(x,y)}{d(x,z)}\right),
\end{equation*}
for all $x,y,z\in X$ with $x\ne z$. A function $f$ is called a \emph{quasisymmetry} if it is an $\eta$-quasisymmetry for some homeomorphism $\eta\colon [0,\infty)\to[0,\infty)$. We will often use the term \emph{distortion function} for homeomorphisms $\eta\colon [0,\infty)\to[0,\infty)$. An important subclass of quasisymmetries are those with $\eta(t)=C\max\{t^{\beta},t^{\frac{1}{\beta}}\}$, for some $0<\beta\leq1$, which we call \emph{($\beta$-)power quasisymmetries}.

Quasisymmetries generalize bi-Lipschitz maps by roughly preserving relative sizes and shapes of sets with similar size that are close together, but allowing sets that are either well separated or that have wildly different sizes to be distorted in different ways. Classifying metric spaces up to quasisymmetric equivalence is a central problem in geometric function theory, and in this context, quasisymmetric invariants play an important role. Conformal dimension, which was introduced by Pansu in \cite{Pansu}, is one of the best known such invariants, see also \cite{bourdonpajot}. 

\subsection{Conformal dimension}
Unlike bi-Lipschitz maps, quasisymmetries can alter the Hausdorff dimension of sets. In fact, since the identity map from $(X,d)$ to its snowflake $(X,d^{\varepsilon})$ with $0<\varepsilon<1$ is a quasisymmetry, and since snowflaking the metric changes the Hausdorff dimension by a factor of $1/\varepsilon$, the Hausdorff dimension of a metric space with positive Hausdorff dimension can be made arbitrarily large with quasisymmetries. On the other hand for some spaces, like $\R^d$, Hausdorff dimension cannot be lowered by quasisymmetries, which motivates the definition of \emph{conformal (Hausdorff) dimension}, defined for a metric space $X$ by
\begin{equation*}
    \Cdimh X\coloneqq \inf\{\dimh f(X)\colon f \text{ is a quasisymmetry}\}.
\end{equation*}
Calculating the conformal dimension of a given metric space is a challenging problem and even determining whether Hausdorff dimension can be lowered by a quasisymmetry is often highly non-trivial. Spaces whose dimension cannot be lowered by quasisymmetries are called \emph{minimal} for conformal dimension, and the prototypical examples are sets of the form $K\times [0,1]$, where $K$ is a compact subset of $\R^d$, see \cite[Proposition 4.1.11]{MackayTyson2010}. Totally disconnected examples were given in \cite{BishopTyson}.

One can of course replace Hausdorff dimension in the definition with other notions of dimension to obtain a family of quasisymmetric invariants. In addition to the conformal Hausdorff dimension, much attention has been given to the \emph{conformal Assouad dimension}, defined by replacing the Hausdorff dimension in the definition above, by the Assouad dimension
\begin{equation*}
    \dima X\coloneqq \inf\Bigg\{s>0\colon \exists C>0,\,\forall 0<r<R,x\in X,\,N_r(X\cap B(x,R))\leq C\left(\frac{R}{r}\right)\Bigg\},
\end{equation*}
where $N_r(A)$ denotes the smallest number of open balls of radius $r>0$ needed to cover $A$. Conformal Assouad dimension is often easier to handle than conformal Hausdorff dimension, see e.g. \cite{KeithLaakso2004,Muruganconf}, and for regular enough spaces, namely quasiself-similar spaces or CLP-spaces, these notions were recently shown to coincide by the second author \cite{Eriksson-Bique2024}. A continuum of other variants of conformal dimension were recently studied in \cite{FraserTyson2025}.

\subsection{Conformal dimension of random fractals}
Recently some progress has been made in understanding the quasiconformal geometry of random objects. In \cite{RossiSuomala2021}, Rossi and Suomala studied quasisymmetric mappings on the classical fractal percolation, which is a random subset of the unit cube $[0,1]^d$ constructed by dividing it into $N^d$ subcubes of side length $1/N$, retaining each independently with probability $p$ and discarding with probability $1-p$ and repeating the process inside all retained subcubes \emph{ad infinitum}. A key feature of the fractal percolation model is that the distribution of the number of offspring at each step forms a \emph{Galton-Watson process} and in particular, the expected number of offspring of each cube is the same at each step of the construction; see \cite[Section 5]{LyonsPeres2016} for more background on such processes. Rossi and Suomala showed that classical fractal percolation is, almost surely conditioned on non-extinciton, \emph{not} minimal for the conformal dimension, see also \cite{FraserTyson2025} for discussion on how other variants of conformal dimension behave for fractal percolation. As pointed out by the authors, the main result of \cite{RossiSuomala2021} easily extends to many other random fractals with underlying Galton-Watson processes. On the other hand, not many examples of random spaces which are minimal for conformal dimension are known. 

This work can be viewed as a natural extension of \cite{RossiSuomala2021}: what happens if we break the Galton-Watson process underlying the fractal percolation and instead allow the expected number of offspring to vary with each construction step? There are essentially two ways to achieve this phenomenon. Firstly, we may vary the retention parameter at each construction step, that is pick a sequence of probabilities $p_n$ and retain each subcube of a cube at level $n$ of the construction with probability $p_n$. Secondly, we may keep the retention probability fixed and vary the number of cubes in the subdivision at each construction step, that is pick a sequence of natural numbers $N_n$ and divide each retained cube of level $n$ into $N_{n+1}^d$ subcubes and retain each with probability $p$. Next we describe these models in detail.

\section{Fat and dense fractal percolations}\label{sec:fat-percolation}
\begin{figure}
    \centering
    \includegraphics[width=0.3\linewidth]{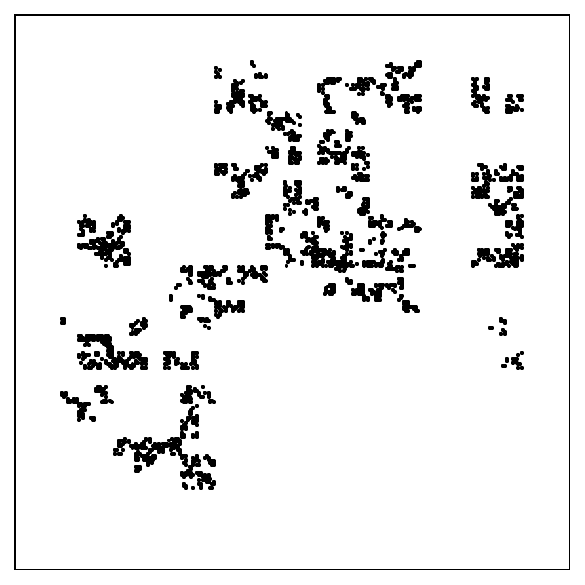}
    \hspace{0.2cm}
    \includegraphics[width=0.3\linewidth]{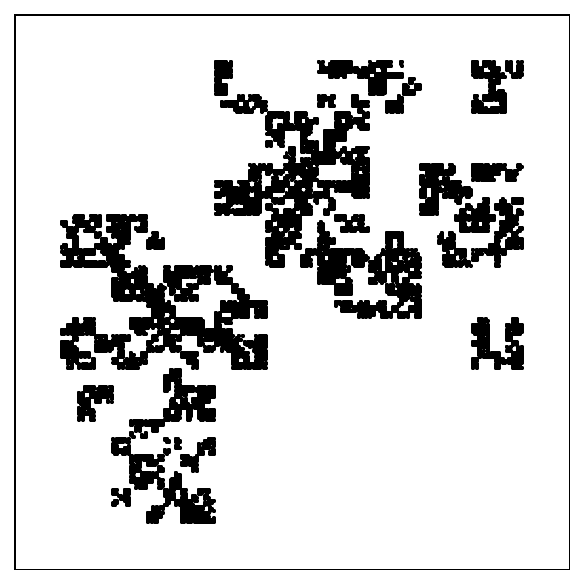}
    \hspace{0.2cm}
    \includegraphics[width=0.3\linewidth]{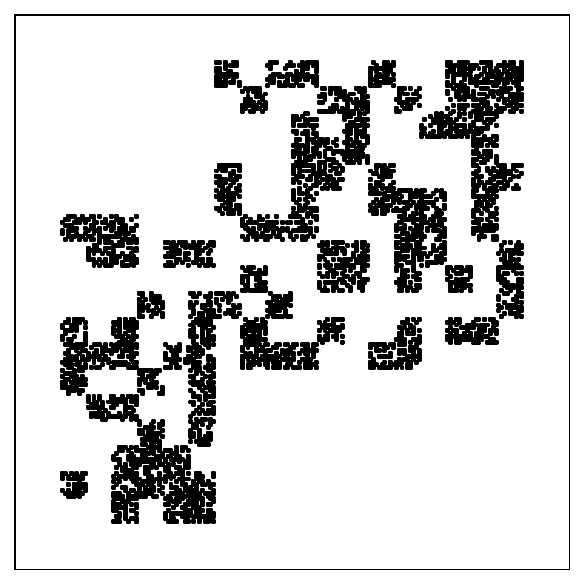}
    \caption{On the left a realization of the $(3,0.5)$-fractal percolation, in the middle, a realization of the $(3, \mathbf{p})$-fat fractal percolation with $p_0 = 0.5$, $p_1 = 0.6$, $p_2 = 0.65,\ldots$, and on the right a realization of the $(\mathbf{N},0.5)$-dense fractal percolation with $N_1=3$, $N_2=6$, $N_3=9,\ldots$.}
    \label{fig:fatpercolation}
\end{figure}

\subsection{Fat fractal percolation}
Fix an integer $N\geq 2$ and a sequence $\textbf{p} := (p_n)_{n\in\N}$ such that $0\leq p_n\leq 1$ for all $n$, and construct a random set $F = F(N, \textbf{p}) \subseteq [0,1]^d$ as follows: Divide the cube $[0,1]^d$ into $N^d$ congruent subcubes, retain each one with probability $p_1$ and discard it with probability $1-p_1$. Denote the set of surviving cubes by $\F_1$. For $n\geq 0$ and a set $\F_n$ of surviving subcubes as above, we define the set $\F_{n+1}$ by repeating the above procedure for each cube in $\F_n$, with the parameter $p_1$ replaced by $p_{n+1}$, and let $\F_{n+1}$ denote the set of all surviving subcubes of side length $N^{-(n+1)}$ obtained this way. Finally, we set 
\begin{equation*}
    F = \bigcap_{n\in\N} \bigcup_{Q\in \F_n} Q\subseteq\R^d
\end{equation*}
and call the set $F$ the $(N, \textbf{p})$-fractal percolation; see Figure \ref{fig:fatpercolation}. Motivated by \cite{RossiSuomala2021}, it is natural to ask the following.
\begin{question}
    Let $N\geq 2$. For which sequences $\textbf{p}$ is the $(N, \textbf{p})$-fractal percolation minimal for conformal dimension with positive probability?
\end{question}

%We remark that only the case when $\textbf{p}$ contains a subsequence converging to $1$ is of interest; If there exists $p<1$ such that $p_n \leq p$ for each $n$, then it easily follows from the result of \cite{RossiSuomala2021} for the $(N,p)$-fractal percolation that $F(N, \textbf{p})$ is not minimal for the conformal dimension, almost surely conditioned on non-extinction {\color{red}Is this clear? What is the Hausdorff dimension for a general sequence $p_n$?}.

A subclass of the random sets $F(N,\textbf{p})$ which has attracted some interest is that of the sequence $\textbf{p}$ converging to $1$; Given a sequence $\textbf{p} = (p_n)_{n\in\N}$ with $\lim_{n\to\infty} p_n = 1$, the set $F = F(N, \textbf{p})$ is called a \emph{$(N,\mathbf{p})$-fat fractal percolation}. Conditioned on the non-extinction of $F$, it is not difficult to see that $\dimh F = d$ almost surely, regardless of the choice of the converging sequence $\textbf{p}$. Nevertheless, the speed of this convergence has significant implications to finer geometric properties of $F$, of which we list a few below. All of the claims hold almost surely conditioned on non-extinction of $F$.
\begin{enumerate}
    \item If $\prod_{n\in\N} p_n^{N^{dn}} = 0$, then $F$ has empty interior.
    \item If $\prod_{n\in\N} p_n >0$, then $F$ has positive Lebesgue measure.
    \item If $\prod_{n\in\N} p_n^{N^{dn}} > 0$, then $F$ is a finite union of closed cubes.
\end{enumerate}
These points are contained in \cite[Theorem 1.9]{BromanEtAl2012}. We refer to \cite{BromanEtAl2012, ChayesPemantlePeres1997} and the references therein for more geometric and topological properties of $F$. Of these known properties of $F$, only the point (3) above ensures that $F$ is minimal for conformal dimension. In fact, for any $d\geq 1$, there exists a set $A\subseteq \R^d$ with $\lambda^d(A) = 1$ and $\Cdimh A = 0$, where $\lambda^d$ denotes the $d$-dimensional Lebesgue measure; see \cite{Romney2019, Tukia1989}. Our first main result demonstrates that regardless of the speed of convergence, fat fractal percolation is almost surely minimal for conformal dimension.
\begin{theorem}\label{thm:main1}
    Let $N\geq 2$. For any sequence $\textbf{p} = (p_n)_{n\in\N}$ such that $\lim_{n\to\infty} p_n = 1$, we have
    \begin{equation*}
         \Cdimh F = \dimh F = d,
    \end{equation*}
    almost surely conditioned on the non-extinction of $F = F(N, \textbf{p})$.
\end{theorem}

\subsection{Dense fractal percolation}

Another natural variant of the fractal percolation is constructed as follows: Let $0<p<1$ and $\mathbf{N}=(N_n)_n$ be a sequence of integers with $N_n\geq 2$ for all $n$. Construct a random set $E = E(\mathbf{N}, p) \subseteq [0,1]^d$ as follows: Divide the cube $[0,1]^d$ into $N_1^d$ congruent subcubes, retain each one with probability $p$ and discard it with probability $1-p$. Denote the set of surviving cubes by $\E_1$. For $n\geq 0$ and a set $\E_n$ of surviving subcubes as above, we define the set $\E_{n+1}$ by repeating the above procedure for each cube in $\E_n$, with the parameter $N_1$ replaced by $N_{n+1}$, and let $\E_{n+1}$ denote the set of all surviving subcubes of side length $\prod_{k=1}^{n+1}N_k^{-1}$ obtained this way. Finally, we set 
\begin{equation*}
    E = \bigcap_{n\in\N} \bigcup_{Q\in \E_n} Q\subseteq\R^d
\end{equation*}
and call the set $E$ the $(\mathbf{N},p)$-fractal percolation; again see Figure \ref{fig:fatpercolation}. If the sequence $(N_n)_n$ is increasing, we call the set $E$ a \emph{$(\mathbf{N},p)$-dense fractal percolation}.

\begin{question}
    Let $0<p<1$. For which sequences $\textbf{N}$ is the $({\bf N}, p)$-fractal percolation minimal for conformal dimension with positive probability?
\end{question}

It turns out that the major difference between the dense and the fat fractal percolation is that the former may contain, at many construction levels, holes, whose size is much larger than the cubes at the next level of the construction. This causes some technical difficulties in adapting the proof of Theorem \ref{thm:main1} to this setting. In particular, we were only able to prove that the Hausdorff dimension of the dense fractal percolation cannot be lowered by power quasisymmetries. This is our second main result.
\begin{theorem}\label{thm:main2}
    Let $p>0$, $\mathbf{N}=(N_n)_{n\in\N} \subset \N$ be an increasing sequence and $E = E(\mathbf{N}, p)$ be the corresponding dense fractal percolation. Almost surely conditioned on the non-extinction of $E$, we have
    \begin{equation*}
        \dimh f(E) = \dimh E = d,
    \end{equation*}
    for any power quasisymmetry $f\colon E\to f(E)$.
\end{theorem}

\subsection{On the proofs}
The proofs of Theorems \ref{thm:main1} and \ref{thm:main2} follow a similar idea, with the proof of Theorem \ref{thm:main2} being somewhat more involved. A substantial portion of the work is done in a deterministic setting, and Theorems \ref{thm:main1} and \ref{thm:main2} follow from the deterministic results by showing that realizations of the percolation processes almost surely, conditioned on non-extinction, contain suitable deterministic subsets.

The key phenomenon we exploit in the proof of Theorem \ref{thm:main1} is that even though the Hausdorff dimension of a given set can often be decreased by enlarging the ``holes'' in the set at all scales by a properly chosen quasisymmetry, an $\eta$-quasisymmetry for a \emph{fixed} $\eta$ cannot enlarge \emph{small} holes by too much. An illustrating example is the following: While any set $E\subset [0,1]$ with $\dimh E<1$ has conformal dimension $0$, for any \emph{fixed} distortion function $\eta$ and any $0<s<1$, there exists a set $E\subset [0,1]$ with $\dimh E<1$, such that $\dimh f(E)\geq s$ for \emph{any} $\eta$-quasisymmetry $f: E\to f(E)$, see Proposition \ref{prop-fatcantordimension}.

Similarly, the proof of Theorem \ref{thm:main2} relies on finding suitable thick enough subsets in the dense fractal percolation. It turns out that with a very large probability, the largest hole in a level $n$ cube $Q$ in the dense fractal percolation is no larger than $\frac{\log N_{n+1}}{N_{n+1}}$ in relative size. These holes are small enough, that we get nice estimates on the the image of $E\cap Q$ under a quasisymmetry $f$ at many scales between levels $n$ and $n+1$, see Lemma \ref{lemma-sub-additive}. For power quasisymmetries, we get additional control on the remaining scales, which is enough to show that the Hausdorff dimension cannot be lowered by $f$. The deterministic results we need are proved in Section \ref{sec-deterministic}.

To show that the fat and dense fractal percolations contain suitable deterministic subsets with probability one, in Section \ref{sec-branching} we adapt results on the existence of $k$-ary subtrees inside Galton-Watson trees from \cite{Chayes,Pakes}. The proofs of Theorems \ref{thm:main1} and \ref{thm:main2} are then finished in Section \ref{sec-proofs}.

\subsection{Notation}
We denote the diameter of a subset $A$ of a metric space by $|A|$. We leave the dependence on the metric, which should be clear from the context, implicit. The distance between sets $A$ and $B$ is denoted by $\dist(A,B)\coloneqq \inf\{d(x,y)\colon x\in A,\,y\in B\}$. If $E\subset \R^2$, $f:E\to f(E)$ is a mapping into an arbitrary metric space and $Q\subset \R^2$, we often write $f(Q)$ for $f(Q\cap E)$ to slightly simplify notation. If $\mathcal{A}$ is any collection of subsets of a metric space $X$ and $B\subset X$, we let $\mathcal{A}(B)=\{A\in\mathcal{A}\colon A\subset B\}$. Given a set $A$ and functions $f,g\colon A\to \R$, we write $f\lesssim g$ if there exists a constant $C$, such that $f(a)\leq Cg(a)$ for all $a\in A$. Similarly, we write $f\gtrsim g$ if $g\lesssim f$. Often $A=\N$ and the constant $C$ may depend on all other quantities except for the indices $a\in\N$. Finally, we let $\pi:\R^d\to\R^{d-1}$ denote the orthogonal projection to the first $d-1$ coordinates. 

\section{Fat and dense Cantor sets}\label{sec-deterministic}
In this section we study quasisymmetric mappings on two slightly different but related classes of deterministic fractals: fat Cantor sets and dense Cantor sets. An observant reader might conjecture that these have something to do with fat and dense fractal percolations, and they would be correct. Indeed, in Section \ref{sec-proofs} the proofs of our main results lean on showing that almost all realizations of fat and dense fractal percolations contain large fat and dense Cantor sets, respectively.

Let us record two simple but crucial lemmas, starting with \cite[Theorem 2.5]{MackayTyson2010}.

\begin{lemma}\label{lemma-boundeddistortion}
Let $f: X \to Y$ be a $\eta$-quasisymmetry. If $A\subseteq B\subseteq X$ are sets with $0<|A|\leq|B| <\infty$, then $|f(B)|<\infty$ and
\begin{equation*}
    2^{-1}\eta\left(\frac{|B|}{|A|}\right)^{-1} \leq \frac{|f(A)|}{|f(B)|} \leq \eta\left(2\frac{|A|}{|B|}\right).
\end{equation*}
\end{lemma}
An application of this lemma gives the following variant, which is useful if $X$ only contains relatively small gaps, i.e. if $|A\cup B|\approx |A|\approx |B|$.
\begin{lemma}\label{lemma-boundeddistortion-dist}
Let $f: X \to Y$ be a $\eta$-quasisymmetry and let $A,B\subset X$ be compact. Then
\begin{equation*}
    \frac{\dist(f(A),f(B))}{|f(A)\cup f(B)|} \leq \eta\left(2\frac{\dist(A,B)}{|A\cup B|}\right).
\end{equation*}
\end{lemma}
We also record the following variant, which will prove useful when the gaps in $X$ are relatively large, but uniformly distributed.
\begin{lemma}\label{lemma-boundeddistortion-variant}
    Let $f\colon X\to Y$ be a $\eta$-quasisymmetry. Let $A,B\subset X$ be non-empty and compact and assume that $|B|\leq |A|$. Then
    \begin{equation*}
        \frac{\dist(f(A),f(B))}{|f(A)|}\leq 1+\eta\left(2+\frac{\dist(A,B)}{|A|}\right).
    \end{equation*}
\end{lemma}
\begin{proof}
    Using compactness, choose $x,y\in A$ satisfying $d(x,y)=|A|$ and $x'\in A$, $z\in B$ satisfying $\rho(f(x'),f(z))=\dist(f(A), f(B))$. Then
    \begin{align*}
        \frac{\dist(f(A),f(B))}{|f(A)|}&\leq \frac{\rho(f(x'),f(x))+\rho(f(x),f(z))}{|f(A)|}\leq 1+\frac{\rho(f(x),f(z))}{\rho(f(x),f(y))}\\
        &\leq 1+\eta\left(\frac{d(x,z)}{d(x,y)}\right)\leq 1+\eta\left(\frac{|A|+|B|+\dist(A,B)}{|A|}\right)\\
        &\leq 1+\eta\left(2+\frac{\dist(A,B)}{|A|}\right).
    \end{align*}
\end{proof}

\subsection{Fat Cantor sets}\label{sec-fatcantor}
In this section, we study quasisymmetries on \emph{fat Cantor sets}. These sets have a very uniform structure: they only contain relatively small and uniformly distributed holes at all scales. The sets will play a crucial role in the proof of Theorem \ref{thm:main2} in Section \ref{sec-proofs}, where we will be able to find large copies of them inside typical realizations of the fat fractal percolation process.

The construction goes as follows: Divide the unit square $[0,1]^d$ into $N^{md}$ congruent subcubes of side length $N^{-m}$, remove one of them arbitrarily, and call the collection of retained subcubes $\F_1$. Given $\F_n$ for $n\geq 0$, divide each cube of $\F_n$ into $N^{md}$ congruent subcubes and again remove one of them. Call $\F_{n+1}$ the family of all retained subcubes of each cube of $\F_n$, and define the $(N,m)$\emph{-fat Cantor set} $F$ by setting
\begin{equation*}
    F = \bigcap_{n\geq 0}\bigcup_{Q\in \F_n} Q.
\end{equation*}
For $n<m$, and $Q\in \F_n$, we let
\begin{equation*}
    \F_{m}(Q)=\{Q'\in \F_m\colon Q'\subset Q\}.
\end{equation*}

\begin{proposition}\label{prop-fatcantordimension}
    Let $\eta$ be a distortion function and $0<\alpha<d$. Then for any large enough $m\in\N$, any $(N,m)$-fat Cantor set $F$ and any $\eta$-quasisymmetry $f: F\to f(F)$, we have
    \begin{equation*}
        \dimh f(F) \geq \alpha.
    \end{equation*}
\end{proposition}
\begin{proof}
    
    Let $\eta$ be a distortion function and $0<\alpha<d$. Our first aim is to show that there exists $m=m(\eta,\alpha)$, such that the following holds: If $F$ is a $(N,m)$-fat Cantor set and $f\colon F\to f(F)$ is an $\eta$-quasisymmetry, then for any $n\in\N$ and $Q_0\in \F_n$, we have
    \begin{equation}\label{eq-sub-additive}
        |f(Q_0)|^{\alpha}\leq \sum_{Q\in \F_{n+1}(Q_0)}|f(Q)|^{\alpha}.
    \end{equation}
    Let $m$ be a natural number which we will make larger when necessary. Let $F$ be a $(N,m)$-fat Cantor set and fix $Q_0\in \F_n$. To simplify notation, we write $\F_{n+1}=\F_{n+1}(Q_0)$. Consider $\R^{d-1}$ as the subspace of $\R^d$ spanned by the first $d-1$ vectors from the natural basis of $\R^d$ and let $\pi\colon \R^{d}\to \R^{d-1}$ denote the orthogonal projection from $\R^d$ onto $\R^{d-1}$. We adopt the convention $\R^0 = \lbrace 0\rbrace$. Define
    \begin{equation*}
        \mathcal{Q}_{n+1}^{d-1}=\{\pi(Q)\colon Q\in \F_{n+1}\},
    \end{equation*}
    which is a collection of $N^{(n+1)m}$-adic cubes in $\R^{d-1}$. Since $\#\F_{n+1}=N^{dm}-1$, we have that $\#\mathcal{Q}_{n+1}^{d-1}=N^{m(d-1)}$. For each $D\in \mathcal{Q}_{n+1}^{d-1}$, let
    \begin{equation*}
        \F_{n+1}(D)=\{Q\in \F_{n+1}\colon \pi(Q)=D\}.
    \end{equation*}

    Now let us fix $D\in \mathcal{Q}_{n+1}^{d-1}$ and enumerate $\F_{n+1}(D)=\{Q_1,\ldots,Q_k\}$ in an ascending order, where $k$ is either $N^{m}$ or $N^{m}-1$. Notice that for all large enough $m$, we have $|\bigcup_{i=1}^kQ_i\cap F|\geq \tfrac{1}{2}N^{-nm}$ and therefore
    \begin{equation*}
        \frac{|\bigcup_{i=1}^kf(Q_i)|}{|f(Q_0)|}\geq \frac{1}{2}\eta\left(\frac{|Q_0\cap F|}{|\bigcup_{i=1}^kQ_i\cap F|}\right)^{-1}\geq \frac{1}{2}\eta(2)^{-1}.
    \end{equation*}
    For large enough $m$, we see from the definition of the $(N,m)$-fat Cantor set that if $Q_i$ and $Q_{i+1}$ are adjacent cubes in $\F_{n+1}(D)$ then $\dist(Q_i\cap F,Q_{i+1}\cap F)\leq 4N^{-(n+2)m}$, except when there is a missing cube of level $(n+1)m$ between $Q_{i}$ and $Q_{i+1}$, in which case $\dist(Q_{i}\cap F,Q_{i+1}\cap F)\leq 4N^{-(n+1)m}$. Since the second case is the worst case scenario, we will assume that there is an index $1<j<k$, where the second case happens. Since $\eta$ is increasing and $|(Q_i\cup Q_{i+1}) \cap F|\geq N^{-(n+1)m}$, we have by Lemma \ref{lemma-boundeddistortion} that
    \begin{equation*}
        \frac{\dist(f(Q_i),f(Q_{i+1}))}{|f(Q_i)\cup f(Q_{i+1})|}\leq\eta\left(2\frac{\dist(Q_{i}\cap F,Q_{i+1}\cap F)}{|(Q_i\cup Q_{i+1}) \cap F|}\right)\leq \eta(8N^{-m})\leq \eta(16N^{-m}),
    \end{equation*}
     for all $i\ne j$ and similarly, since $|Q_0\cap F|\geq \tfrac{1}{2}N^{-mn}$,
    \begin{equation*}
        \frac{\dist(f(Q_j),f(Q_{j+1}))}{|f(Q_0)|}\leq \eta\left(\frac{2\dist(Q_j\cap F,Q_{j+1}\cap F)}{|Q_0\cap F|}\right) \leq \eta(16N^{-m}).
    \end{equation*}
    By noticing that $|f(Q_i)\cup f(Q_{i+1})|\leq |f(Q_i)|+|f(Q_{i+1})|+\dist(Q_i,Q_{i+1})$ for all $i=1,\ldots,k$, if we take $m$ large enough so that $\eta(16N^{-m})\leq \tfrac{1}{2}$, we have
    \begin{equation*}
        |f(Q_i)\cup f(Q_{i+1})|\leq 2(|f(Q_i)|+|f(Q_{i+1})|),
    \end{equation*}
    for all $i\ne j$. Now we use the fact that $\left|\bigcup_{i=1}^{\ell}A_i\right|\leq \sum_{i=1}^{\ell-1} |A_i\cup A_{i+1}|$ for any collection of sets $\{A_i\}$ so in particular
    \begin{align*}
        \frac{1}{2}&\eta(2)^{-1}|f(Q_0)|\leq \left|\bigcup_{i=1}^{j}f(Q_i)\right|+\left|\bigcup_{i=j+1}^{k}f(Q_i)\right|+\dist(f(Q_{j}),f(Q_{j+1}))\\
        &\leq \sum_{i=1}^{j-1}|f(Q_i)\cup f(Q_{i+1})|+\sum_{i=j+1}^{k-1}|f(Q_i)\cup f(Q_{i+1})|+\eta(16N^{-m})|f(Q_0)|\\
        &\leq 4\sum_{i=1}^k|f(Q_i)|+\frac{1}{4}\eta(2)^{-1}|f(Q_0)|,
    \end{align*}
    where we take $m$ larger if necessary to make sure that $\eta(16N^{-m})\leq \tfrac{1}{4}\eta(2)^{-1}$.
    In other words, for any $D\in\mathcal{Q}_{n+1}^{d-1}$, we have
    \begin{equation}\label{eq-cetabound}
        |f(Q_0)|\leq c_{\eta}\sum_{Q\in \F_{n+1}(D)}|f(Q)|,
    \end{equation}
    where $c_{\eta}=16\eta(2)$.

    Let now $0<\alpha<d$. Since
    \begin{equation*}
        \sum_{D\in \mathcal{Q}_{n+1}^{d-1}}\sum_{Q\in \F_{n+1}(D)}|f(Q)|=\sum_{Q\in \F_{n+1}}|f(Q)|,
    \end{equation*}
    by \eqref{eq-cetabound} and Hölder's inequality, we get
    \begin{align*}
        N^{m(d-1)}|f(Q_0)|&\leq c_{\eta}\sum_{Q\in \F_{n+1}}|f(Q)|\leq c_{\eta}\left(\sum_{Q\in \F_{n+1}}|f(Q)|^{\alpha}\right)^{\frac{1}{\alpha}}N^{\frac{dm(\alpha-1)}{\alpha}}.
    \end{align*}
    Therefore
   \begin{equation*}
        |f(Q_0)|^{\alpha}\leq c_{\eta}^{\alpha}N^{m(\alpha-d)}\sum_{Q\in \F_{n+1}}|f(Q)|^{\alpha},
    \end{equation*}
    and since $c_{\eta}^{\alpha}N^{m(\alpha-d)}\to 0$ as $m\to \infty$, \eqref{eq-sub-additive} follows.

    Now our aim is to define a measure $\mu$ on $f(F)$ such that for any $n\in\N$ and $Q\in \F_n$, 
    \begin{equation*}
    \mu(f(Q)) \leq |f(Q)|^{\alpha}.
    \end{equation*}
    This will be enough to prove the claim, since it is easy to see that using standard methods, see e.g. \cite[Lemma 4.5]{Hakobyan2009}, that $\mu$ is an $\alpha$-Frostman measure on $f(F)$ and therefore $\dimh f(F) \geq \alpha$. 
    
    Let $\mu(f(F)) = 1$, and for each $n\in\N$ and $Q_n\in \F_n$, define
    \begin{equation}\label{eq-measure-def}
        \mu(f(Q_n)) = \frac{|f(Q_n)|^\alpha}{\sum_{Q \in \F_{n+1}(Q^*)} |f(Q)|^\alpha} \mu(f(Q^*)) 
    \end{equation}
    where $Q^*\in \F_{n-1}$ is the unique cube containing $Q_n$. By Caratheodory's extension theorem, this defines a probability measure $\mu$ on $f(F)$. 
    
    Now for any $Q_n\in \F_n$, if $Q_{\ell}\in \E_{\ell}$, $\ell=0,\ldots, n-1$ are the unique cubes such that $Q_{n}\subseteq Q_{n-1}\subseteq\cdots\subseteq Q_0 = [0,1]^d$, we have
    \begin{align*}
        \frac{\mu(f(Q_n))}{|f(Q_n)|^\alpha} &= \prod_{\ell=1}^n \frac{|f(Q_{\ell-1})|^\alpha}{ \sum_{Q_* \in \F_{n+1}(Q_{\ell-1})} |f(Q_*)|^\alpha}.
    \end{align*}
    It follows from \eqref{eq-sub-additive} that
    \begin{equation*}
    \prod_{\ell=1}^n \frac{|f(Q_{\ell-1})|^\alpha}{ \sum_{Q_* \in \F_{n+1}(Q_{\ell-1})} |f(Q_*)|^\alpha}\leq1,
    \end{equation*}
    which finishes the proof
\end{proof}

\subsection{Dense Cantor sets}\label{sec-uf-flat-cantor}
In this section, we modify the results of the previous section for a related class of fractals we call \emph{dense Cantor sets}. The idea is similar to the construction of fat Cantor sets, but the uniformly distributed holes in these sets are allowed to be much larger compared to the size of the cubes in the construction, and in particular, the ratio of the size of the holes to the size of the construction cubes is allowed to grow to infinity at a controlled rate. Unfortunately, for precisely this reason, we are only able to extend Proposition \ref{prop-fatcantordimension} for power quasisymmetries in this setting.

Let $(N_n)_n$ be an increasing sequence of integers and assume that $N_1\geq 2$. For simplicity we assume that $N_n=2^{k_n}$ are dyadic, where $k_n$ is an increasing sequence of integers. Let $\mathcal{Q}_1$ denote the partition of the unit square $Q_0\subset \R^d$ to $N_1^d$ congruent subcubes of side length $N_1^{-1}$. Let $\E_1\subset \mathcal{Q}_1$ denote an arbitrary subcollection of the cubes. Divide each cube $Q\in \E_1$ to $N_2^d$ congruent subcubes of side length $N_1^{-1}N_2^{-1}$, choose an arbitrary sub collection denoted by $\E_2(Q)$ and let $\E_2=\bigcup_{Q\in \E_1}\E_2(Q)$. Continue this process indefinitely and let
\begin{equation*}
    E=\bigcap_{k=1}^{\infty}\bigcup_{Q\in\E_k}Q.
\end{equation*}
For a sequence $(\Delta_n)_{n\in\N}$ with $\Delta_n>0$, we call a set $E$ as above a \emph{$(\Delta_n)_n$-dense Cantor set} if for all large enough $n\in\N$ and every $Q\in\E_n$, 
\begin{equation}\label{eq-uniformly-dense}
    \sup\left\{|A|\colon A\subset Q\setminus \bigcup \E_{n+1}\ \text{is a line segment parallel to}\right\} \le \Delta_{n+1}|Q|.
\end{equation}
Recall that here $\pi:\R^d\to\R^{d-1}$ denotes the orthogonal projection to the first $d-1$ coordinates. Informally, \eqref{eq-uniformly-dense} means that the union of descendants of $Q$ contains no ``vertical'' gaps of diameter $\Delta_{n+1}|Q|$. In this section, we prove the following proposition.
\begin{proposition}\label{prop-dense-cantor-space}
    Let $E\subset \R^d$ be a $\left(\frac{\log N_{n}}{N_{n}}\right)_n$-dense Cantor set and let $f\colon E\to f(E)$ be a power quasisymmetry. Then
    \begin{equation*}
        \dimh f(E)=d.
    \end{equation*}
\end{proposition}
This will follow by constructing for each $0<\alpha<d$ a Frostman measure on $f(E)$ similarly as in the proof of Proposition \ref{prop-fatcantordimension}, once we establish a suitable analogue of \eqref{eq-sub-additive}. There are additional technical difficulties in this setting compared to the setting of the previous section, mainly arising from the fact that we have no control over the difference between two scales that follow each other in the construction. This is a technical problem in the construction of the Frostman measure, since we want to control the measure of cubes at all scales, not just at the construction scales. However, if the cubes at level $n+1$ of the construction are much smaller than the cubes at level $n$, the gaps between adjacent level $n+1$ cubes are small enough so that we can establish \eqref{eq-sub-additive} on a large number of scales between levels $n$ and $n+1$. This is the content of the following lemma.
\begin{lemma}\label{lemma-sub-additive}
    Let $\eta$ be a distortion function, $0< \alpha< d$ and $E$ be a $\left(\frac{\log N_{n}}{N_{n}}\right)_n$-dense Cantor set. Let $f\colon E\to f(E)$ be an $\eta$-quasisymmetry. Then for any $0<\gamma<1$, for all large enough $n$, if $Q\in\mathcal{D}_k(E)$ with $\sum_{m=1}^nk_{m}\leq k < \sum_{m=1}^{n}k_{m}+\gamma k_{n+1}$, then
    \begin{equation*}
        |f(Q)|^\alpha \leq \sum_{Q'\in \mathcal{D}_{k+1}(E\cap Q)} |f(Q')|^\alpha.
    \end{equation*}
\end{lemma}
\begin{proof}
    Fix $\beta>0$, $0< \alpha < d$ and let $0<\gamma<1$. Let $E\subset\R^d$ be a $\left(\frac{\log N_{n}}{N_{n}}\right)_n$-dense Cantor set and let $Q\in \mathcal{D}_k$, where $\sum_{m=1}^nk_{m}\leq k < \sum_{m=1}^{n}k_{m}+\gamma k_{n+1}$. 
    %It follows from \eqref{eq-uniformly-dense} with $\Delta_n=\log N_n/N_n$, together with the assumption on $k$, that
    %\begin{equation*}
        %|\pi^{-1}(D)\cap Q\cap E|\geq \frac{1}{2}|Q|,
    %\end{equation*}
    %for all large enough $n$. 
    Let
    \begin{equation*}
        \mathcal{Q}_{k+1}'(Q)=\{\pi(Q)\colon Q\in\mathcal{D}_{k+1}(E\cap Q)\}.
    \end{equation*}
    Let $D\in\mathcal{Q}_{k+1}'(Q)$ and let $\{Q_1, \ldots, Q_M\}$ denote the enumeration of $\mathcal{E}_{k+1}(D)\coloneqq \{Q\in \mathcal{D}_{k+1}(E\cap Q)\colon \pi(Q)=D\}$ in ascending order. It follows from \eqref{eq-uniformly-dense} and the assumption on $k$ that
    \begin{equation*}
        \dist(Q_i\cap E,Q_{i+1}\cap E)\leq \frac{\log N_{n+1}}{N_{n+1}}2^{-\sum_{m=1}^n k_m}\leq 2^d\frac{\log N_{n+1}}{N_{n+1}^{1-\gamma}}|Q_i|\leq \frac{1}{2}|Q_i|.
    \end{equation*}
    Moreover, it follows from \eqref{eq-uniformly-dense} with $\Delta_n=\log N_n/N_n$ that for all large enough $n$, we have
    \begin{equation*}
        |Q_i\cap E|\geq \frac{1}{2}|Q_i|.
    \end{equation*}
    %details for the above:
    %\begin{align*}
       % |Q_i\cap E|&\geq |Q_i|-2\frac{\log N_{n+1}}{N_{n+1}}2^{-d\sum_{m=1}^nk_m}\\
       % &\geq 2^{-d(k+1)}-2\frac{\log N_{n+1}}{N_{n+1}^{1-\gamma}}2^{-dk}\\
       % &\geq \frac{1}{2}2^{-d(k+1)}=\frac{1}{2}|Q_i|.
    %\end{align*}
    Combining the previous two inequalities and applying Lemma \ref{lemma-boundeddistortion-variant}, we have for all large enough $n$, that
    \begin{align}\label{eq-upperbound}
        |f(Q)| & = \left|\bigcup_{i=1}^Mf(Q_i)\right|\leq \sum_{i=1}^M |f(Q_{i})| + \sum_{i=1}^{M-1}{\dist}(f(Q_i), f(Q_{i+1}))\nonumber\\
        &\leq \left(2+\eta\left(2+4\frac{\log N_n}{N_{n+1}^{(1-\gamma)}}\right)\right)\sum_{i=1}^{M} |f(Q_i)|\leq 3\eta(3)\sum_{Q'\in \E_{k+1}(D)}|f(Q')|.
    \end{align}
    Note that
    \begin{equation*}
        \sum_{D\in \mathcal{Q}'_{k+1}(Q)}\sum_{Q'\in \E_{k+1}(D)}|f(Q)|=\sum_{Q'\in \mathcal D_{k+1}(Q\cap E)}|f(Q')|,
    \end{equation*}
    and by again using \eqref{eq-uniformly-dense} and our assumption on $k$, $\#\mathcal{Q}'_{k+1}(Q)= 2^{(d-1)k}$. If $\alpha<d=1$, then it follows from \eqref{eq-upperbound} and Lemma \ref{lemma-boundeddistortion} that 
    \begin{align*}
        |f(Q)|^\alpha &\leq \ 3\eta(3) \sum_{Q'\in \mathcal D_{k+1}(Q\cap E)}|f(Q')|^\alpha \cdot \left(\frac{|f(Q')|}{|f(Q)|}\right)^{1-\alpha} \\
        &\leq 3\eta(3)\eta(2N_{k+1}^{-1})^{1-\alpha}\sum_{Q'\in \mathcal D_{k+1}(Q\cap E)}|f(Q')|^\alpha
    \end{align*}
    which proves the claim since $N_{k+1}\to \infty$ as $n\to \infty$. If $1<\alpha<d$, then applying Hölder's inequality in addition to \eqref{eq-upperbound} we get
    \begin{align*}
        2^{(d-1)k}|f(Q)|&\leq 3\eta(3)\sum_{Q'\in \mathcal D_{k+1}(Q\cap E)}|f(Q')|\\
        &\leq 3\eta(3)\left(\sum_{Q'\in \mathcal D_{k+1}(Q\cap E)}|f(Q')|^{\alpha}\right)^{\frac{1}{\alpha}}2^{\frac{dk(\alpha-1)}{\alpha}}
    \end{align*}
    which implies
   \begin{equation*}
        |f(Q)|^{\alpha}\leq 3^{\alpha}\eta(3)^{\alpha}2^{(\alpha-d)k}\sum_{Q'\in \mathcal D_{k+1}(Q\cap E)}|f(Q)|^{\alpha}.
    \end{equation*}
    Since $\alpha<d$ and $k\to \infty$ as $n\to\infty$, we get the claim. 
\end{proof}
Note that we did not require $f$ to be a power quasisymmetry in the previous lemma. This assumption comes into play when we want to control the measures of cylinders at scales $\sum_{m=1}^{n}k_{m}+\gamma k_{n+1}<k<\sum_{m=1}^{n+1}k_{m}$.
\begin{proof}[Proof of Proposition \ref{prop-dense-cantor-space}]
As in the proof of Proposition \ref{prop-fatcantordimension}, it suffices to find for for all $0<\alpha<d$, a measure $\mu$ on $f(E)$, which satisfies $\mu(f(Q))\lesssim |f(Q)|^{\alpha}$, for all $Q\in\mathcal{D}(Q)$.

Let $0<\alpha<t<d$, and let $\frac{t}{t+(t-\alpha)\beta^2}<\gamma<1$. Let $n_0$ be large enough, such that Lemma \ref{lemma-sub-additive} holds with the exponent $t$, for all $Q\in\mathcal{D}_k(E)$ with $\sum_{m=1}^nk_{m}\leq k < \gamma \sum_{m=1}^{n+1}k_{m}$ and $n\geq n_0$. Take $Q_0\in\mathcal{D}_{\sum_{m=1}^{n_0}k_{m}}(E)$ and let $\mu(f(Q_0)) = 1$. Construct a measure on $\mu(f(Q_0))$ by setting for each $Q\in\mathcal{D}_k(Q_0)$ with $\sum_{m=1}^{n_0}k_{m}< k < \sum_{m=1}^{n_0}k_{m}+\gamma k_{n_0+1}$,
\begin{equation}\label{eq-measure-def-large-scales}
    \mu(f(Q)) = \frac{|f(Q)|^t}{\sum_{Q \in \mathcal{D}_{k+1}(Q^*\cap E)} |f(Q)|^t} \mu(f(Q^*)) 
\end{equation}
where $Q^*\in \mathcal{D}_{k-1}$ is the unique cube containing $Q$. After this, for $Q\in \mathcal{D}_{\sum_{m=1}^{n_0+1}k_{m}}$, we set
\begin{equation}\label{eq-measure-def-small-scales}
    \mu(f(Q)) = \frac{|f(Q)|^t}{\sum_{Q \in \mathcal{D}_{\sum_{m=1}^{n_0+1}k_{m}}(Q^*\cap E)} |f(Q)|^t} \mu(f(Q^*)),
\end{equation}
where $Q^*\in \mathcal{D}_{k_0'}$ is the unique cube containing $Q$ and $k_0'$ is the largest integer smaller than $\sum_{m=1}^{n_0}k_{m}+\gamma k_{n+1}$. Continue dividing mass with this process: for every $\ell\in \N$,  distribute mass on $Q\in\mathcal{D}_{k}$ for $\sum_{m=1}^{n_0+\ell}k_{m}< k < \sum_{m=1}^{n_0+\ell}k_{m}+\gamma k_{n_0+\ell+1}$ as in \eqref{eq-measure-def-large-scales} and then skip to scale $ \sum_{m=1}^{n_0+\ell+1}k_{m}$ and use \eqref{eq-measure-def-small-scales}. We now claim that for any $k\geq \sum_{m=1}^{n_0}k_m$ and $Q\in\mathcal{D}_{k}(Q_0)$, we have
\begin{equation}\label{eq:alpha-frostman}
    \mu(f(Q))\leq |f(Q)|^{\alpha}.
\end{equation}
If $Q\in\mathcal{D}_k$ for some $\sum_{m=1}^{n_0+\ell}k_{m}\leq k < \sum_{m=1}^{n_0+\ell}k_{m}+\gamma k_{n_0+\ell+1}$ and $\ell\in\N$, then \eqref{eq:alpha-frostman} indeed holds with $t$ in place of $\alpha$ (and thus in particular with the exponent $\alpha$), with the same argument as in the proof of Proposition \ref{prop-fatcantordimension}.

It remains to prove \eqref{eq:alpha-frostman} for $Q\in\mathcal{D}_k$ with $\sum_{m=1}^{n_0+\ell}k_{m}+\gamma k_{n_0+\ell+1}< k \leq \sum_{m=1}^{n_0+\ell+1}k_{m}$ and $\ell\in\N$. Let $Q_1\in\mathcal{D}_{k'}$ be the unique cube satisfying $Q\subset Q_1$, where $k'$ is the largest integer smaller than $\sum_{m=1}^{n_0+\ell}k_{m}+\gamma k_{n_0+\ell+1}$. Since $|Q| \geq 2^{-\sum_{m=1}^{n_0+\ell+1}k_{m}}$, Lemma \ref{lemma-boundeddistortion} asserts that
\begin{align*}
    \mu(f(Q))&\leq |f(Q_1)|^{t}\lesssim \eta\left(\frac{2^{-\sum_{m=1}^{n_0+\ell}k_m-\gamma k_{n_0+\ell+1}}}{2^{-\sum_{m=1}^{n_0+\ell+1}k_m}}\right)^t|f(Q)|^t\\
    &\lesssim N_{n_0+\ell+1}^{\frac{t(1-\gamma)}{\beta}}|f(Q)|^t.
\end{align*}
Furthermore, Lemma \ref{lemma-boundeddistortion} also implies that
\begin{equation}\label{eq-ineq-not-optimal}
    |f(Q)|\lesssim |f([0,1]^d)|N_{n_0+\ell+1}^{-\gamma\beta},
\end{equation}
whence
\begin{equation*}
    |f(Q)|^{\alpha-t}\gtrsim N_{n_0+\ell+1}^{\gamma\beta(t-\alpha)}.
\end{equation*}
Since $\gamma >\frac{t}{t+(t-\alpha)\beta^2}$, the preceding inequalities give
\begin{equation*}
    \mu(f(Q))\lesssim |f(Q)|^{\alpha-t}|f(Q)|^{t}=|f(Q)|^{\alpha},
\end{equation*}
finishing the proof.
\end{proof}

\section{Branching processes}\label{sec-branching}

In this section, we describe the probabilistic framework which allows us to find large subsets of the fractal percolations. Whereas the offspring distribution of the classical fractal percolation model follows a simple Galton-Watson process, we require a more general framework.

For every $n\in\N$, let us fix a finite index set $\Gamma_n$, with $\#\Gamma_n=N_n$, and denote by $\mathcal{P}_n$ the set of all probability measures supported on the collection of subsets of $\Gamma_n$. Our standing assumptions are that $(N_n)_{n\in\N}$ is a non-decreasing sequence of integers and $N_1\geq 2$. We define a \emph{branching process} as follows. Start with a root node $\varnothing$ and pick an arbitrary probability measure $\eta_{\varnothing}\in\mathcal{P}_1$. Generate \emph{children} for $\varnothing$ by drawing a random subset $\Gamma_{\varnothing}\subset \Gamma_1$ with respect to $\eta_{\varnothing}$. For each $i\in\Gamma_{\varnothing}$ pick a probability measure $\eta_{i}\in\mathcal{P}_2$ and draw a subset $\Gamma_i\subset \Gamma_2$ with respect to $\eta_i$. Continue this process iteratively: for every $\mtt{i}\coloneqq i_1i_2\cdots i_n\in\Gamma_\varnothing\times \Gamma_{i_1}\times \Gamma_{i_1i_2}\times\ldots\times \Gamma_{i_1\cdots i_{n-1}}$ and $j\in\Gamma_{\mtt{i}}$, pick a probability measure $\eta_{\mtt{i}j}\in\mathcal{P}_{n+1}$ and draw a random subset $\Gamma_{\mtt{i}j}\subset\Gamma_{n+1}$  with respect to $\eta_{\mtt{i}j}$. An alternative way to view the process is to start by fixing an arbitrary collection of probability measures
\begin{equation*}
    H\coloneqq \bigcup_{k=0}^{\infty}\{\eta_{i_1\cdots i_k}\in\mathcal{P}_k\colon i_1\in \Gamma_1,\ldots,i_k\in \Gamma_k\},
\end{equation*}
and running the process described above for this fixed collection. For a fixed $H$  we call
\begin{equation*}
    T(H)\coloneqq\{i_1i_2\cdots\in\Gamma_{1}\times\Gamma_{2}\times\cdots\colon i_k\in \Gamma_{i_1\cdots i_{k-1}}\forall k\in\N\},
\end{equation*}
the \emph{family tree} of the branching process. Given a sequence $p_n$ with $0<p_n<1$ for all $n$, we call the tree $T(H)$ \emph{$(p_n)_n$-thick} if, for all $n$,
\begin{equation}\label{eq-np-admissible}
    \PP_{i_1\cdots i_n}(\#\Gamma_{i_1\cdots i_n}=N_{n+1})\geq p_n
\end{equation}
where $\PP_{i_1\cdots i_n}$ is the law of $\#\Gamma_{i_1\cdots i_n}$. In our applications, the collection of probability measures $H$ is left implicit, and since the definition of $(p_n)_n$-thickness only depends on the laws $\PP_{i_1\cdots i_n}$, going forward, we suppress the dependence on $H$ from the notation and instead, call a tree $T$ $(p_n)_n$-thick if for each $i_1\cdots i_n$, the law $\PP_{i_1\cdots i_n}$ satisfies \eqref{eq-np-admissible}.

We say that the tree $T$ contains a \emph{$(N_{n}-1)_n$-subtree} if for all $i_1i_2\cdots\in T$ and $n=0,1,\ldots$,
\begin{equation*}
    \#\Gamma_{i_1\cdots i_n}\geq N_{n+1}-1,
\end{equation*}
that is, all nodes of level $n$ have at least $N_{n+1}-1$ children. The following result, which is a generalization of \cite[Theorem 5.29]{LyonsPeres2016}, is the main result of this section.

\begin{theorem}\label{thm-completesubtree}
    If $T$ is $(1-N_{n}^{-6})_n$-thick, then there exists $p_0>0$ depending only on $(N_n)_{n\in\N}$ such that
    \begin{equation*}
        \PP(T\text{ contains a $(N_{n}-1)_n$-subtree})\geq p_0.
    \end{equation*}
\end{theorem}
\begin{proof}
Consider a random process where we mark each child of a node $\mtt{i}$ of level $n$ independently, and the probability of marking any given child is at least $1-s$. Since $T$ is $(1-N_{n}^{-6})_n$-thick, the probability that $\mtt{i}$ has at most $N_{n+1}-2$ marked children is at most
\begin{align*}
    \sum_{k=0}^{N_{n+1}}&\PP_{\mtt{i}}(\#\Gamma_{\mtt{i}}=k)\sum_{j=0}^{N_{n+1}-2}\binom{k}{j}(1-s)^js^{k-j}\\
    &\leq\sum_{k=0}^{N_{n+1}-1}N_{n+1}^{-6}\sum_{j=0}^{k}\binom{k}{j}(1-s)^js^{k-j}+\sum_{j=0}^{N_{n+1}-2}\binom{N_{n+1}}{j}(1-s)^js^{N_{n+1}-j}\\
    &\leq \sum_{j=0}^{N_{n+1}-2}\binom{N_{n+1}}{j}(1-s)^js^{N_{n+1}-j}+N_{n+1}^{-5}\\
    &=1-(1-s)^{N_{n+1}}-N_{n+1}s(1-s)^{N_{n+1}-1}+N_{n+1}^{-5}\eqcolon g_{n+1}(s).
\end{align*}

Let $q_n$ denote the probability that $T$ does not contain a $(N_{n}-1)_n$-subtree of height $n$, and $\tau := \lim_{n\to\infty} 1-q_n$ the probability that $T$ contains a full $(N_{n}-1)_n$-subtree. By marking a child of $\varnothing$ if it is retained, we see using the above that $q_1\leq g_1(0)$. Furthermore, by marking a child of $\varnothing$ if it has at least $N_2-1$ children, we have $q_2\leq g_1(g_2(0))$. By iterating this process---marking a child of $\varnothing$ if it has at least $N_2-1$ children which have at least $N_3-1$ children, and so on---with similar reasoning we have
\begin{equation*}
    q_n \leq g_1\circ g_2\circ\cdots\circ g_n(0).
\end{equation*}
It remains to show that there exists a constant $c<1$, such that $g_1\circ g_2\circ\cdots\circ g_n(0)\leq c$, for all $n\in\N$. Simple computations show that $g'_k(0) = 0$, $g_k^{(\ell)}(0) \leq N_k^{\ell+1}$ for $2\leq \ell \leq N_k$ and $g_k^{(\ell)}(0)= 0$ for $\ell > N_k$. Thus by Taylor's theorem at $0$, for any $k\in\N$ and $s\in[0,N_k^{-1}/2)$,
\begin{equation}\label{eq:taylors-thm}
    g_k(s)\leq N_k^{-5}+ N_k \sum_{\ell=2}^{N_k} (N_k s)^\ell \leq N_k^{-5} + 2 N_k^3 s^2.
\end{equation}
%\begin{equation}\label{eq:taylors-thm}
%  g_k(s) \leq N_k^{-5} + \frac{N_k^2}{2} s^2 + \frac{N_k^4}{6} s^3
%\end{equation}

We now claim that
\begin{equation*}
    g_{k}\circ g_{k+1}\circ\ldots\circ g_n(0)\leq 4N_k^{-5},
\end{equation*}
for all $k=1,\ldots,n$. This follows by induction starting from $k=n$, since clearly $g_n(0)=N_n^{-5}\leq4 N_n^{-5}$, and if 
\begin{equation*}
    g_{k+1}\circ g_{k+2}\circ\ldots\circ g_n(0)\leq 4N_{k+1}^{-5},
\end{equation*}
then by \eqref{eq:taylors-thm},
\begin{equation*}
    g_{k}\circ g_{k+1}\circ\ldots\circ g_n(0)\leq N_k^{-5}+8N_{k}^{3} N_{k+1}^{-10}\leq 4N_k^{-5}.
\end{equation*}
%\begin{equation*}
%    g_{k}\circ g_{k+1}\circ\ldots\circ g_n(0)\leq N_k^{-5}+2N_{k+1}^{-8}+\frac{8}{6}N_{k+1}^{-11}\leq 2N_k^{-5}.
%\end{equation*}
Since $N_1\geq 2$, this gives the claim with $c=\frac{1}{8}$.
\end{proof}

\section{Proofs of Theorems \ref*{thm:main1} and \ref*{thm:main2}}\label{sec-proofs}
In this section we prove our main results. When we say that a set $E\subset \R^d$ \emph{contains} a set $A$, we mean that there exists a homothety $h\colon \R^2\to\R^2$, such that $h(A)\subset E$. Combined with results of Section \ref{sec-fatcantor}, the following result will yield a proof for Theorem \ref{thm:main1}. 
\begin{theorem}\label{thm-fatsubsets}
    Let $\mathbf{p}\coloneqq (p_n)_{n=0}^{\infty}$ be such that $\lim_{n\to\infty}p_n=1$. Then for any $m\in\N$, the set $F(N,\mathbf{p})$ contains a $(N,m)$-fat Cantor set, almost surely conditioned on non-extinction.
\end{theorem}
\begin{proof}
    Let $F=F(N,\mathbf{p})$. Notice first that for a fixed $m\in\N$, it suffices to show that there exists $p_0>0$, such that for any large enough $n\in\N$, and any $N$-adic cube $Q_0$ of level $n$,
    \begin{equation}\label{eq-posprob}
        \mathbb{P}(F\cap Q_0\text{ contains a $(N,m)$-fat Cantor set}\cond Q_0\in\mathcal{F}_n)\geq p_0.
    \end{equation}
    This can be seen to imply the claim, for example, by a simple application of Markov's inequality and the Borel-Cantelli lemma, see e.g. \cite[Lemma 3.1]{SuomalaShmerkin2020}
    
    Fix an integer $m\in\N$ and let $\varepsilon>0$ and let $n\in\N$ be large enough, such that
    \begin{equation*}
        \prod_{\ell=1}^m p_{n+k+\ell-1}^{N^{\ell d}}\geq 1-N^{-6md},
    \end{equation*}
    for all $k\in\N$. Consider a branching process with a root node $Q_0\in\F_n$ and children $\F_{n+m}(Q_0)$. For each $Q\in \F_{n+m}(Q_0)$ we consider the cubes in $\F_{n+2m}(Q)$ the children of $Q$ and denote the tree obtained by iterating this process with $T(Q_0)$. Evidently $F\cap Q_0$ contains a $(N,m)$-fat Cantor set if and only if the tree $T(Q_0)$ contains a complete $N^{md}-1$-ary subtree. Since for any $k\in\N$ and $Q\in\F_{n+km}(Q_0)$,
    \begin{equation*}
        \PP(\#\F_{n+(k+1)m}(Q_0)= N^m)=\prod_{\ell=1}^m p_{n+km+l-1}^{N^{\ell d}}\geq 1-N^{-6md},
    \end{equation*}
    the tree $T(Q_0)$ is $(1-N^{-6md})$-thick, and \eqref{eq-posprob} follows from Theorem \ref{thm-completesubtree}.
\end{proof}

The following corollary is immediate.
\begin{corollary}
    For any $\textbf{p} = (p_n)_{n\in\N}$ such that $\lim_{n\to\infty}p_n=1$, almost surely conditioned on non-extinction, the set $E_{\textbf{p}}$ contains a $(N,m)$-fat Cantor set for all $m\in\N$.
\end{corollary}
We are now ready to prove Theorem \ref{thm:main1}.
\begin{proof}[Proof of Theorem \ref{thm:main1}]
 Let $F=F(N,\textbf{p})$ be a realization of the fat fractal percolation with a sequence of probabilities $\textbf{p} = (p_n)_{n\in\N}$ satisfying $\lim_{n\to \infty}p_n=1$, which contains a $(N,m)$-fat Cantor set for all $m\in\N$. By the previous corollary, conditioned on non-extinction, this is an event of probability one. Let $\eta$ be a distortion function, $0<\alpha<d$ and $f\colon F\to f(F)$ be a $\eta$-quasisymmetry. Let $m\in\N$ be large enough, such that Proposition \ref{prop-fatcantordimension} holds for the $(N,m)$-fat Cantor set $F_m\subset F$. Since $f|_{F_m}$ is an $\eta$-quasisymmetry, Proposition \ref{prop-fatcantordimension} gives
 \begin{equation*}
     \dimh f(F)\geq \dimh f(F_m)\geq \alpha,
 \end{equation*}
 which gives the claim by taking $\alpha\to d$.
\end{proof}

Let us proceed to the proof of Theorem \ref{thm:main2}. We will use the following analogue of Theorem \ref{thm-fatsubsets}, which immediately implies Theorem \ref{thm:main2} by Proposition \ref{prop-dense-cantor-space}.

\begin{theorem}
    Let $p>0$ and $\mathbf{N}=(N_n)_n$ be an increasing sequence of natural numbers. There exists a constant $C>0$ such that almost surely conditioned on non-extinction, the set $E(\mathbf{N},p)$ contains a $\left(C\tfrac{\log N_{n+k}}{N_{n+k}}\right)_k$-dense Cantor set for some $n \in \N$.
\end{theorem}
\begin{proof}
    Again, we note that it suffices to show that there exists $C, p_0>0$, such that for any large enough $n\in\N$ and any cube $Q_0$ with side length $\prod_{m=1}^nN_m^{-1}$,
    \begin{equation*}
        \mathbb{P}(E\cap Q_0\text{ contains a $\left(C\tfrac{\log N_{n+k}}{N_{n+k}}\right)_k$-dense Cantor set}\cond Q_0\in\mathcal{E}_n)\geq p_0.
    \end{equation*}

    Let us denote by $\overline{\log}$ the logarithm in base $(1-p)^{-1}$. Choose a surviving cube $Q_0\in\mathcal{E}_n$, and consider the following branching process conditional on $Q_0 \in \mathcal E_n$. 
    
    Set $Q_0$ as the root node of the branching process, and retain $Q_0$ if
    \begin{align}\label{eq-small-holes}
        &\sup\left\{|A|\colon A\subset Q_0\setminus \bigcup\mathcal{E}_{n+1}\ \text{is a line segment parallel to}\ \ker \pi\right\}\nonumber \\
        &\leq \frac{6d\,\overline{\log} N_{n+1}}{N_{n+1}}|Q_0|,
    \end{align}
    and otherwise discard $Q_0$. Informally, $Q_0$ is retained if any only if there is no ``vertical'' block of $6d\overline{\log}N_{n+1}$-many cubes removed from $Q_0$ in the next construction step of $E$. If $Q_0$ is retained, let $\E_{1}' := \lbrace Q \in \E_{n+1}:\ Q\subset Q_0\rbrace$ denote the descendants of $Q_0$. 
    
    For each $k > 1$ with $\E_{k-1}'$ defined, for each cube $Q\in\mathcal{E}_{k-1}'$ we repeat the process, that is, retain $Q$ if and only if no ``vertical'' block of $6d\overline{\log}N_{n+k}$-many cubes is removed from $Q_0$ in the next construction step, that is,
    \begin{align}\label{eq-small-holes2}
        &\sup\left\{|A|\colon A\subset Q\setminus \bigcup\mathcal{E}_{n+k}\ \text{is a line segment parallel to}\ \ker \pi\right\} \nonumber\\
        &\leq \frac{6d\,\overline{\log} N_{n+k}}{N_{n+k}}|Q|.
    \end{align}
    Denote the descendants of $Q$ by $\mathcal{E}_{k}'(Q) = \lbrace Q' \in \E_{n+k}:\ Q'\subset Q\rbrace$. By trimming the collections $\E_{k}'(Q)$ for every $Q \in \E_{k-1}$ while retaining the property \eqref{eq-small-holes2} and replacing $6$ by $13$, we may suppose that $\#\E_{k}'(Q)$ is constant, denoted by $M_k$. Note that as a formal consequence of \eqref{eq-small-holes2}, we may do the trimming in a way that $N_{n+k}^{d-1} \frac{N_{n+k}}{6d \overline{\log} N_{n+k}} \leq M_k \leq 2N_{n+k}^{d-1} \frac{N_{n+k}}{6d \overline{\log} N_{n+k}}$, in particular, $(M_k)_{k\in\N}$ is non-decreasing once $n$ is large enough. Let $\E_{k}' = \bigcup_{Q\in \E_{k-1}'} \E_k'(Q)$. Having defined $\E_k'$ for every $k>1$ in this way, let
    \begin{equation*}
        E'=\bigcap_{k> 1}\bigcup_{Q\in\E'_k}Q.
    \end{equation*}
    Next we observe that conditioned on $\mathcal E_{k}'$ for $k>1$, the probability that $Q \in \mathcal E_k'$ is not retained (that is, \eqref{eq-small-holes2} fails and a ``vertical'' block of at least $6d \overline{\log} N_{n+k+1}$-many cubes is removed from $Q$) is at most
    \begin{equation*}
        N_{n+k+1}^d(1-p)^{6d \overline{\log} N_{n+k+1}}\leq N_{n+k+1}^{-5d}\leq  M_{n+k+1}^{-5}
    \end{equation*}
    Therefore the family tree of the branching process which defines $E'$ is $(1-M_{n+k+1}^{-5})_{k\in\N}$-thick and thus contains a $(M_{k}-1)_k$-subtree with probability at least $p_0>0$ by Theorem \ref{thm-completesubtree}. We denote the subset of $E'$ corresponding to this subtree by $E'' = \bigcap_{k>1} \bigcup_{Q\in \E''_k} Q$. Since the sets $\E''_k$ are formed by removing at most one cube from each set $\lbrace Q \in \E'_k:\ Q\subset Q'\rbrace$, $Q'\in \E'_{k-1}$, it follows that for each $k>1$ and $Q\in \E''_k$, no ``vertical'' block of $26d\,\overline{\log} N_{n+k+1}+1$-many cubes is removed from $Q$ in the next construction step, that is,
   \begin{align*}
        &\sup\left\{|A|\colon A\subset Q\setminus \bigcup\mathcal{E}''_{n+k+1}\ \text{is a line segment parallel to}\ \ker \pi\right\} \nonumber\\
        &\leq \frac{26d\,\overline{\log} N_{n+k+1} + 1}{N_{n+k+1}}|Q| \leq \frac{27d\,\overline{\log} N_{n+k+1}}{N_{n+k+1}}|Q|.
    \end{align*}
    In particular, the set $E''$ is a $(\frac{27d \overline{\log} N_{n+k}}{N_{n+k}})_k$-dense Cantor set. This proves the statement with $C = -27d\log(1-p)$.
\end{proof}

\bibliographystyle{abbrv}
\bibliography{Bibliography}

\begin{thebibliography}{10}

\bibitem{BishopTyson}
C.~J. Bishop and J.~T. Tyson.
\newblock Locally minimal sets for conformal dimension.
\newblock {\em Ann. Acad. Sci. Fenn. Math.}, 26(2):361--373, 2001.

\bibitem{bourdonpajot}
M.~Bourdon and H.~Pajot.
\newblock Cohomologie {$l_p$} et espaces de {B}esov.
\newblock {\em J. Reine Angew. Math.}, 558:85--108, 2003.

\bibitem{BromanEtAl2012}
E.~I. Broman, T.~van~de Brug, F.~Camia, M.~Joosten, and R.~Meester.
\newblock Fat fractal percolation and {{\(k\)}}-fractal percolation.
\newblock {\em ALEA, Lat. Am. J. Probab. Math. Stat.}, 9(2):279--301, 2012.

\bibitem{Chayes}
J.~T. Chayes, L.~Chayes, and R.~Durrett.
\newblock Connectivity properties of {M}andelbrot's percolation process.
\newblock {\em Probab. Theory Related Fields}, 77(3):307--324, 1988.

\bibitem{ChayesPemantlePeres1997}
L.~Chayes, R.~Pemantle, and Y.~Peres.
\newblock No directed fractal percolation in zero area.
\newblock {\em J. Statist. Phys.}, 88(5-6):1353--1362, 1997.

\bibitem{Eriksson-Bique2024}
S.~Eriksson-Bique.
\newblock Equality of different definitions of conformal dimension for
  quasiself-similar and {CLP} spaces.
\newblock {\em Ann. Fenn. Math.}, 49(2):405--436, 2024.

\bibitem{FraserTyson2025}
J.~Fraser and J.~T. Tyson.
\newblock Sobolev and quasiconformal distortion of intermediate dimension with
  applications to conformal dimension.
\newblock arXiv e-print, available at https://arxiv.org/abs/2505.10525, 2025.

\bibitem{Hakobyan2009}
H.~Hakobyan.
\newblock Conformal dimension: {Cantor} sets and {Fugelede} modulus.
\newblock {\em Int. Math. Res. Not. IMRN}, 2010(1):87--111, 2009.

\bibitem{KeithLaakso2004}
S.~Keith and T.~Laakso.
\newblock Conformal {A}ssouad dimension and modulus.
\newblock {\em Geom. Funct. Anal.}, 14(6):1278--1321, 2004.

\bibitem{LyonsPeres2016}
R.~Lyons and Y.~Peres.
\newblock {\em Probability on trees and networks}, volume~42 of {\em Cambridge
  Series in Statistical and Probabilistic Mathematics}.
\newblock Cambridge University Press, New York, 2016.

\bibitem{MackayTyson2010}
J.~M. Mackay and J.~T. Tyson.
\newblock {\em Conformal dimension}, volume~54.
\newblock American Mathematical Society, Providence, RI, 2010.

\bibitem{Muruganconf}
M.~Murugan.
\newblock Conformal {A}ssouad dimension as the critical exponent for
  combinatorial modulus.
\newblock {\em Ann. Fenn. Math.}, 48(2):453--491, 2023.

\bibitem{Pakes}
A.~G. Pakes and F.~M. Dekking.
\newblock On family trees and subtrees of simple branching processes.
\newblock {\em J. Theoret. Probab.}, 4(2):353--369, 1991.

\bibitem{Pansu}
P.~Pansu.
\newblock Dimension conforme et sph\`ere \`a{} l'infini des vari\'et\'es \`a{}
  courbure n\'egative.
\newblock {\em Ann. Acad. Sci. Fenn. Ser. A I Math.}, 14(2):177--212, 1989.

\bibitem{Romney2019}
M.~Romney.
\newblock Singular quasisymmetric mappings in dimensions two and greater.
\newblock {\em Adv. Math.}, 351:479--494, 2019.

\bibitem{RossiSuomala2021}
E.~Rossi and V.~Suomala.
\newblock Fractal percolation and quasisymmetric mappings.
\newblock {\em Int. Math. Res. Not. IMRN}, (10):7372--7393, 2021.

\bibitem{SuomalaShmerkin2020}
V.~Suomala and P.~Shmerkin.
\newblock Patterns in random fractals.
\newblock {\em Am. J. Math.}, 142(3):683--749, 2020.

\bibitem{Tukia1989}
P.~Tukia.
\newblock Hausdorff dimension and quasisymmetric mappings.
\newblock {\em Math. Scand.}, 65(1):152--160, 1989.

\end{thebibliography}

\end{document}